\documentclass[11pt]{amsart}
\usepackage[english]{babel}
\usepackage[utf8]{inputenc}
\usepackage{cite}
\usepackage{amsmath,amsthm,amsfonts,amssymb,bbm,verbatim}
\usepackage{caption,latexsym,listings,setspace}
\usepackage[top=1.5in, bottom=1.5in, left=1.2in, right=1.1in]{geometry}
\usepackage{lipsum}
\usepackage{lmodern}
\usepackage{hyperref}
\usepackage[savepos,totpages,titleref,dotfill,counter,user]{zref}
\usepackage[usenames,dvipsnames]{pstricks}
\usepackage{epsfig}
\usepackage{pst-all}
\usepackage{isomath}
\newcounter{plainsection}
\renewcommand*{\theplainsection}{\arabic{plainsection}}

\newcounter{resultcounter}[section]
\renewcommand{\thesection}{\S\arabic{section}}

\newcounter{assumptioncounter}
\newcommand*{\asslabel}[1]{(A\refstepcounter{assumptioncounter}\theassumptioncounter\label{#1})}
\newcommand*{\assref}[1]{(A\ref{#1})}

\makeatletter
\zref@newprop{theresult}{\theresultcounter}
\zref@newprop{theplainsection}{\theplainsection}
\zref@newprop{thesection}{\thesection}
\zref@addprop{main}{theresult}
\zref@addprop{main}{thesection}
\zref@addprop{main}{theplainsection}
\newcommand*{\plainsection}{%
  \refstepcounter{plainsection}%
}
\newcommand*{\thmref}[1]{\zref[theplainsection]{#1}.\zref[theresult]{#1}\,}
\makeatother

\newenvironment{thm}[1][]{\refstepcounter{resultcounter}\par\medskip
   \textbf{Theorem~\theplainsection.\theresultcounter.} \textsf{#1}\em}{}
\newenvironment{prop}[1][]{\refstepcounter{resultcounter}\par\medskip
   \textbf{Proposition~\theplainsection.\theresultcounter.} \textsf{#1}\em}{}

\newenvironment{cor}[1][]{\refstepcounter{resultcounter}\par\medskip
   \textbf{Corollary~\theplainsection.\theresultcounter.} \textsf{#1}\em}{}

\newenvironment{defn}[1][]{\par\medskip
   \textbf{Definition.~}\textsf{#1}\em}{}

\newenvironment{rem}[1][]{\refstepcounter{resultcounter}\par\medskip
   \textbf{Remark~\theplainsection.\theresultcounter.} \textsf{#1}\em}{}
\newenvironment{lem}[1][]{\refstepcounter{resultcounter}\par\medskip
   \textbf{Lemma~\theplainsection.\theresultcounter.}  \textsf{#1}\em}{}

\newcommand{\C}{\mathbb{C}}

\newcommand{\R}{\mathbb{R}}

\newcommand{\U}{\mathbf{U}}
\newcommand{\Z}{\mathfrak{Z}}
\newcommand{\p}{\mathbf{p}}

\newcommand{\e}{\varepsilon}

\newcommand{\abs}[1]{\left\vert{#1}\right\vert}  

\newlength{\bibitemsep}
\newlength{\bibparskip}
\let\oldthebibliography\thebibliography
\renewcommand\thebibliography[1]{%
  \oldthebibliography{#1}%
  \setlength{\parskip}{\bibitemsep}%
  \setlength{\itemsep}{\bibparskip}%
}

\fontfamily{lmtt}\selectfont
\begin{document}
\begin{abstract}
We revisit the existence problem of heteroclinic connections in $\R^N$ associated with Hamiltonian systems involving potentials $W:\R^N\to\R$ having several global minima. Under very mild assumptions on $W$ we present a simple variational approach to first find geodesics minimizing length of curves joining any two of the potential wells, where length is computed with respect to a degenerate metric having conformal factor $\sqrt{W}.$ Then we show that when such a minimizing geodesic avoids passing through other wells of the potential at intermediate times, it gives rise to a heteroclinic connection between the two wells. This work improves upon the approach of  \cite{sternberg1991vector} and represents a more geometric alternative to the approaches of e.g. \cite{alikakos2008connection, Bolotin, byeon2016double, rabinowitz1993homoclinic} for finding such connections. \end{abstract}

\title{On the heteroclinic connection problem for multi-well gradient systems}
\author{Andres Zuniga}
\address{Department of Mathematics\\
  Indiana University\\
  Bloomington, IN 47405.}
\email[A.~Zuniga]{ajzuniga@indiana.edu} 

\author{Peter Sternberg}
\email[P.~Sternberg]{sternber@indiana.edu}

\date{4/22/2016}

\maketitle


\section{Introduction}\plainsection\zlabel{intro} 

In this paper we revisit the question of existence of heteroclinic connections associated with multiple-well potentials. 
Given a potential $W:\R^N\to[0,\infty)$ whose zero set $\Z$ consists of $m$ distinct global minima $\p_1,\ldots, \p_m\in\R^N$, with $m\geq 2$, 
we pursue the question of existence of solutions $U:\R\to\R^N$ to the Hamiltonian system
 \begin{equation}\label{ConnectionEqn}
 \begin{array}{l}
  U''-\nabla_u W(U)=0 \quad \mbox{ on }\quad (-\infty,+\infty),\\[0.2cm]
  U(-\infty)=\p_j,\quad U(+\infty)=\p_k,
  \end{array}
 \end{equation} 
connecting any two of the wells $\p_j,\p_k$ with $j\neq k$.

Existence of such vector-valued heteroclinics under a variety of hypotheses on the potential and on the values of $m$ and $N$ has been obtained by a number of authors over the years, including \cite{alikakos2008connection, Bolotin, rabinowitz1993homoclinic}, based on finding critical points or minimizers of the associated Lagrangian
\[
H(U):=\int_{-\infty}^{+\infty} \frac{1}{2}\abs{U'}^2+W(U).
\]

Here we return instead to the approach of  \cite{sternberg1991vector}, originally introduced  in \cite{sternberg1988effect} for a related problem where the potential vanishes along two planar curves. In the case of a planar system $N=2$ and a double-well potential $m=2$ existence was established in \cite{sternberg1991vector} under somewhat stringent non-degeneracy assumptions on the behavior of $W$ near the wells. 
Now we place the existence question within the context of minimizing geodesics in length spaces. Under quite weak assumptions on $W$ near the wells, we provide a simple proof of existence for solutions to \eqref{ConnectionEqn} for $m=2$ and $N$ arbitrary, as well as a geometric characterization of sufficient conditions for existence that hold for any $m\geq 3$ and any $N\geq 2$. The realization of heteroclinic connections as minimizing geodesics, in a sense to be described below, yields a more geometric characterization of these curves in phase space than one typically gets from minimization of $H$.

The considerable interest in heteroclinic connections arises in part from the central role they play in analyzing models for phase transitions, in particular time-dependent and stationary solutions to the so-called vector Allen-Cahn system
\[
u_t=\Delta u -\nabla_uW(u),
\]
and its variants, see e.g., \cite{alama1997stationary, alikakos2012newproof, alikakos2013maximum, alessio2014multiplicity, bates2013entire, bethuel2011slow, bethuel2013slow, bronsard1996three, bronsard1993three, byeon2016double, rabinowitz1993homoclinic, rubinstein1989fast, schatzman2002asymmetric, sternberg1994local}.\\ 

The starting point for the approach here and in \cite{sternberg1991vector} is the observation that heteroclinic connections enjoy the property of equipartition of energy, namely 
\begin{equation}\label{equi}
\int_{\R} \frac{1}{2}|U'|^2= \int_{\R} W(U).
\end{equation}
Consequently, viewing $H(U)$ as a sum of squares, one sees that heteroclinic connections yield equality in the trivial inequality $H(U)\geq \sqrt{2}E(U)$ satisfied by any competitor, where
\begin{equation}\label{Edefn}
E(U):=\int_{\R}\sqrt{W(U)}|U'|.
\end{equation}
This naturally leads one to consider the minimization problem
\begin{equation}
\inf\{E(U): U(-\infty)=\p_j,\;U(+\infty)=\p_k\}\quad\mbox{for}\quad j,k\in\{1,2,\ldots,m\}.\label{Emin}
\end{equation}
We observe that \eqref{Emin} is purely geometric, with the value of $E$ depending only on a curve, not on its parametrization, so one regards this as a problem of minimizing the distance  between $\p_j$ and $\p_k$ in a degenerate Riemannian metric having conformal factor $\sqrt{W}$, a metric denoted here by $d(\p_j,\p_k)$. It follows immediately from the use of an \emph{equipartition parametrization}, i.e. one in which a minimizer of $E$ obeys \eqref{equi}, that a minimizer of $E$ yields a minimizer of $H$, hence a solution to \eqref{ConnectionEqn}. 

The approach in \cite{sternberg1991vector} is to carry out this program, namely finding a minimizer of \eqref{Emin}, for the case of two wells when $U$ is $\R^2$-valued ($m=2$ and $N=2$) by first solving the perturbed problem 
 \[
 \inf E_{\delta}(U)\quad\mbox{with}\quad E_{\delta}(U):=\int_{\R}\left(\sqrt{W(U)}+\delta\right)\abs{U'},
 \]
for $\delta>0$ in which the degeneracy is removed, and then passing to the limit $\delta\to 0$ in the minimizers. Obtaining $\delta$-independent bounds to establish the needed compactness for this procedure, however, is somewhat messy and seems to require rather strong assumptions on the non-degeneracy of the Hessian of $W$ at the points of $\Z$. In the present approach, we work directly in the metric space $\R^N$ endowed with metric $d$ and obtain much more general results with far weaker hypotheses.
 
 In~\zref[thesection]{sec:existence-geodesics} we pursue the question of existence of minimizers of \eqref{Emin} under very mild assumptions on $W:\R^N\to \R$, basically just continuity and non-zero behavior at infinity. We first introduce a notion of length  in the metric space $(\R^n,d)$, cf. \eqref{Ldefn}, and establish the equivalence of length of a curve and its $E$ value, cf. Theorem~\thmref{thm:EequalsL}. Then we show the existence of minimizing geodesics joining any two points in $\R^N$,  cf. Theorem~\thmref{thm:existence-length-min}.
 
In Section~\zref[thesection]{sec:study-connection-problem} we exhibit conditions under which minimizers of $E$ yield minimizers of $H$, hence solutions to \eqref{ConnectionEqn}, cf. Theorem~\thmref{thm:existence-connection}. This naturally requires further regularity assumptions on $W$ beyond continuity so as to make \eqref{ConnectionEqn} meaningful. When there are three or more potential wells, then making this logical bridge between $E$ minimizers and $H$ minimizers requires an additional assumption, namely that the minimizing geodesic joining two wells by solving \eqref{Emin} does not pass through any other wells on its way. This is precisely the obstruction to existence of heteroclinic connections that the authors of \cite{alama1997stationary} first revealed for certain planar systems and which was also examined in detail in \cite{alikakos2006explicit,alikakos2008connection} when the potential takes the form $W(z)=\abs{f(z)}^2$ where $f$ is holomorphic. Here we establish this necessary condition for non-existence for very general $W$, any $m\geq 3$, and $N$ arbitrary. 
\vskip.1in\noindent
{\bf Acknowledgments.} The authors wish to thank Jiri Dadok for pointing out the approach to finding geodesics via length spaces. P.S. wishes to acknowledge the support of the National Science Foundation through D.M.S. 1362879.

\medskip

\section{Existence of minimizing geodesics}\plainsection\zlabel{sec:existence-geodesics}

In this section we establish the existence of curves solving the problem \eqref{Emin}. Our approach leads us into the realm of length spaces.

\subsection{Geometric framework}\zlabel{subsec:geometric-framework}

For $N\geq 2$, we take $W:\R^N\to[0,\infty)$ to be any continuous function satisfying the properties below
\medskip
\begin{itemize}
\item[\asslabel{asswells}] The zero set $\Z$ of $W$ is given by $m$ distinct points $\Z=\{\p_1,\ldots \p_m\}$ so that $W(\p_1)=\ldots =W(\p_m)=0$, and $W>0$ elsewhere.\\[-0.2cm]
\item[\asslabel{assdecay}] $\liminf_{|p|\to\infty}W(p)>0$.\\[-0.2cm]
\end{itemize}

In order the make the notation somewhat simpler, we will work in this section with $F:=\sqrt{W}$ rather than $W$. 
Throughout, we will write $B_{r}(p):=\{x\in\R^N:|x-p|<r\}$ for the Euclidean ball centered at $p\in\R^N$ with radius $r>0$.
Given any continuous curve $\gamma:[0,1]\to \R^N$ we denote the set of times at which the curve runs into the zeros of $F$ by
\begin{equation}\label{TimeSet}
\mathcal{T}^{\Z}_{\gamma}:= \{t\in [0,1]: \gamma(t)\in \Z\}.
\end{equation}

A central role in our analysis will be played by the set of continuous curves defined on $[0,1]$ whose restrictions to connected sub-arcs that have no intersection with the zeros of $F$ are locally Lipschitz continuous in $\R^N$, endowed with the standard Euclidean metric $|\cdot|$. We denote this class of curves by
\[
Lip_{\Z}([0,1];\R^N):= \{\gamma\in C^0([0,1];\R^N): \,\, \gamma\in Lip_{loc}(([0,1]\setminus \mathcal{T}^{\Z}_{\gamma});\R^N) \}.
\]
We remark that in the special case where $\gamma\in Lip_{\Z}([0,1];\R^N)$ is such that $\mathcal{T}^{\Z}_{\gamma}=\emptyset$, then if fact $\gamma$ is a Lipschitz continuous curve.  Also, for any two distinct points $p,q\in \R^N$, we denote 
\[
Lip_{\Z}(p,q):= \{\gamma\in Lip_{\Z}([0,1];\R^N)\,| \,\, \gamma(0)=p,\, \gamma(1)=q\}.
\]
We consider now the functional $E:Lip_{\Z}([0,1];\R^N)\to \R$ that is used to define a notion of length of curves in $Lip_{\Z}([0,1];\R^N)$, using a metric conformal to the standard Euclidean one with $F$ as a degenerate conformal factor: 
\[
 E(\gamma):=\int_{0}^1F(\gamma(t))\abs{\gamma'(t)}dt.
\]
\begin{defn}
Let us introduce a metric $d$ on $\R^N$ induced by the functional $E$ by letting
\begin{equation}\label{Emini}
d(p,q):=\inf_{\gamma\in Lip_{\Z}(p,q)}E(\gamma),\quad \mbox{ for any }p,q\in\R^N.
\end{equation}
This metric gives rise to a natural length structure associated to it, by means of
\begin{equation}
L(\gamma):=\sup_{\{t_j\}^N_{j=1}\in P([0,1])} \sum^N_{j=1} d(\gamma(t_j),\gamma(t_{j+1})),\label{Ldefn}
\end{equation}
where $P([0,1])$ is the set of finite partitions of $[0,1]$. The value $L(\gamma)$ will be called the length of a curve $\gamma$, and we will say that a curve $\gamma$ is $E$-rectifiable when it has finite length $L(\gamma)<\infty$.
\end{defn}
\vskip.1in
Despite the degeneracy of $F$, it is easy to check that $d$ satisfies the properties of a metric on $\R^N$. 
It is worth mentioning that $d$ so defined makes $(\R^N,d)$ into a \emph{length space}, in the sense that for the metric space $(\R^N,d)$, the value of $d(p,q)$ is equal to the infimum of the length of admissible curves joining $p$ and $q$, see \cite[pp. 32]{bridson1999metric}.\\ 

Before proceeding, we make note of the easy inequality
\begin{equation}
d(p,q)\leq L(\gamma)\quad\mbox{for all}\;p,q\in\R^N\;\mbox{and all curves}\;\gamma\in Lip_{\Z}(p,q),\label{dlessL}
\end{equation}
that follows immediately from \eqref{Ldefn} by choosing the partition $P=\{0,1\}$ of $[0,1]$.
\subsection{Equivalence of $E(\gamma)$ and $L(\gamma)$ for any curve $\gamma\in Lip_\Z$}\zlabel{subsec:basic-prop-length-structure}
\medskip

Our first goal is to establish the equivalence of $E(\gamma)$ and $L(\gamma)$. To this end, we begin with a standard lower-semi-continuity property of the length functional in general length spaces, see e.g.~\cite{bridson1999metric}. For the sake of completeness, however, the proof is included.

\begin{lem}\zlabel{lem:lsc-length-functional} 
Let $\{\gamma_n\}$ be a sequence of curves from $[0,1]$ to $\R^N$, converging uniformly to an $E$-rectifiable curve $\gamma_0$ in the $d$ metric. Then 
\[
\liminf_{n\to\infty}L(\gamma_n)\geq L(\gamma_0).
\]
\end{lem}
\begin{proof}[Proof of Lemma~\thmref{lem:lsc-length-functional}]
Consider an $E$-rectifiable curve $\gamma_0:[0,1]\to \R^N$ and suppose $\gamma_n\overset{d}{\rightrightarrows} \gamma_0$. 
Let $\{t_j\}^J_{j=0}$ be any partition in $P([0,1])$. The uniform convergence yields for any $\e>0$, $\exists n_0(\e)$ so that for $n\geq n_0(\e)$
\[
\sup_{t\in [0,1]} d(\gamma_n(t),\gamma_0(t))<\frac{\e}{2J}.
\]
It follows that for $n\geq n_0(\e)$, and all $0\leq j\leq J$:

\[
\begin{array}{rl}
d(\gamma_0(t_j),\gamma_0(t_{j+1}))\leq & 
d(\gamma_0(t_j),\gamma_n(t_j))+
d(\gamma_n(t_j),\gamma_n(t_{j+1}))+
d(\gamma_n(t_{j+1}),\gamma_0(t_{j+1}))\\[0.2cm]
\leq & {\displaystyle \frac{\e}{2J}+d(\gamma_n(t_j),\gamma_n(t_{j+1}))+\frac{\e}{2J}}.
\end{array}
\]
Adding these inequalities over $j\in\{0,\ldots, J\}$ we deduce 
\[
\sum^J_{j=0} d(\gamma_0(t_j),\gamma_0(t_{j+1}))\leq L(\gamma_n)+\e,
\]
from which it follows $\sum^J_{j=0} d(\gamma_0(t_j),\gamma_0(t_{j+1}))\leq 
\liminf\limits_{n\to\infty} L(\gamma_n)+\e$. Taking the supremum over all partitions in $P([0,1])$ we conclude that $L(\gamma_0)\leq \liminf\limits_{n\to\infty} L(\gamma_n)+\e$, with $\e>0$ arbitrary.
\end{proof}

\medskip


Next, we establish a lower-semi-continuity property of the functional $E$.

\begin{lem}\zlabel{lem:lsc-E-functional} Consider $\{\gamma_j\},\,\gamma_0\in Lip_{\Z}([0,1];\R^N)$.  
Then the functional $E:Lip_{\Z}([0,1];\R^N)\to \R$ is lower semi-continuous with respect to  uniform convergence in the Euclidean metric:
\[
\liminf\limits_{j\to\infty}E(\gamma_j) \geq E(\gamma_0)
\quad \mbox{ whenever } \quad \sup_{t\in[0,1]}\abs{\gamma_j- \gamma_0}\rightarrow 0\,\,\mbox{  as }\,\, j\to\infty.
\]
\end{lem}
\begin{proof}[Proof of Lemma~\thmref{lem:lsc-E-functional}]
With no loss of generality we assume $\liminf_{j\to\infty}E(\gamma_j)<\infty$. Fixing $\e>0$, let us consider the punctured plane $\Omega_{2\e}=\R^N\setminus (\bigcup^m_{l=1}\{x:|x-\p_l|<2\e\})$, and in addition we introduce the set
$T_{\e}=\{t\in [0,1]: \gamma_0(t)\in \Omega_{2\e}\}$ with the possibility that $[0,1]\setminus T_\e$ could be empty if $\gamma_0$ avoids $\Z$.
The uniform convergence $\gamma_j\rightrightarrows \gamma_0$ yields the existence of a value $j_0(\e)$ such that 
$\bigcup_{j\geq j_0(\e)}\{\gamma_j(t): t\in T_{\e}\}\subset\Omega_{\e}$. 
Next, we decompose
\begin{equation}\label{LscEeqn1}
\int_{T_{\e}}F(\gamma_j)|\gamma'_j|\,dt= 
\int_{T_{\e}}(F(\gamma_j)-F(\gamma_0))|\gamma'_j|\,dt
+\int_{T_{\e}}F(\gamma_0)|\gamma'_j|\,dt.
\end{equation}
The key observation is that $\{\gamma_j\}$ restricted to $T_{\e}$ has bounded Euclidean arc-length: 
\[ 
\liminf\limits_{j\to\infty}\int_{T_{\e}}|\gamma'_j|\, dt\leq 
\frac{1}{c_{\e}}\liminf \limits_{j\to\infty}\int_{T_{\e}}F(\gamma_j)|\gamma'_j|\,dt
\leq \frac{1}{c_{\e}}\liminf\limits_{j\to\infty}E(\gamma_j)\equiv C_{\e}<\infty.
\]
for $c_{\e}=\min_{p\in{\Omega}_{\e}} F(p)>0$. In particular, upon the extraction of a subsequence, we may assume $\|\gamma'_j\|_{L^1(T_{\e})}\leq C_{\e}$ for all $j\geq j_0(\e)$.  By virtue of the uniform continuity of $F$ on the compact set 
$\{p: \inf_{t\in [0,1]} |p-\gamma_0(t)|\leq \delta\}$ for some $\delta=\delta(\e)$, together with the uniform convergence $\gamma_j\rightrightarrows \gamma_0$ in the Euclidean metric, we get the bound 
 $\max_{t\in [0,1]} |F(\gamma_j)-F(\gamma_0)|<\e/C_{\e}$ for all $j\geq j_1(\e)$.
Whence, choosing $j\geq \max\{j_0(\e),j_1(\e)\}$ it follows that
\begin{equation}\label{LscEeqn2}
\limsup_{j\to\infty}\abs{\int_{T_{\e}}(F(\gamma_j)-F(\gamma_0))|\gamma'_j|\,dt}\leq \e.
\end{equation}
For the second term in \eqref{LscEeqn1}, note that $t\mapsto F(\gamma_0(t))$ is a continuous positive function defined on $T_{\e}$, thus we can use the following characterization of bounded variation for an $L^1$-function:
\[
\int_{T_{\e}}F(\gamma_0)|\gamma'_j|\,dt
=\sup_{\substack{g\in C^1_c(T_{\e};\R^N)\\ |g(t)|\leq F(\gamma_0(t)) \text{ for }t\in T_{\e} }}
\int_{T_{\e}}\gamma_j(t)\cdot g'(t)\,dt.
\]
Fix now such a vector field $g$. The uniform convergence $|\gamma_j-\gamma_0|\rightrightarrows 0$ implies 

\[
\liminf_{j\to\infty}\int_{T_{\e}}F(\gamma_0)|\gamma'_j|\,dt\geq 
\liminf_{j\to\infty}\int_{T_{\e}}\gamma_j(t)\cdot g'(t)\,dt
=\int_{T_{\e}}\gamma_0(t)\cdot g'(t)\,dt,
\]
so taking the supremum over $g\in C^1_c(T_{\e};\R^N)$ with $|g(t)|\leq F(\gamma_0(t))$ on $T_{\e}$ we arrive at
\begin{equation}\label{LscEeqn3}
\liminf_{j\to\infty}\int_{T_{\e}}F(\gamma_0)|\gamma'_j|\,dt
\geq\int_{T_{\e}}F(\gamma_0)|\gamma'_0|\,dt.
\end{equation} 
Applying the estimates \eqref{LscEeqn2} and \eqref{LscEeqn3} to the identity \eqref{LscEeqn1}, one derives 
\[ 
\liminf_{j\to\infty}\int_{T_{\e}}F(\gamma_j)|\gamma'_j|\,dt
\geq \int_{T_{\e}}F(\gamma_0)|\gamma'_0|\,dt-\e.
\]
Now, the continuity of $\gamma_0$ ensures the convergence of the characteristic functions $\chi_{T_{\e}}\to \chi_{[0,1]}\equiv 1$ for a.e.~$t\in[0,1]$, as $\e \to 0^+$.  Then the monotone convergence theorem applied to the above inequality proves the desired conclusion
\[
\begin{array}{rl}
 \liminf\limits_{j \to\infty}E(\gamma_j) \geq & {\displaystyle \limsup_{\e \to 0^+} \left(\int_{T_{\e}}F(\gamma_0)|\gamma'_0|\,dt-\e\right)}\\[0.4cm]
=& {\displaystyle \int_{[0,1]}F(\gamma_0)|\gamma'_0|\,dt=E(\gamma_0)}.\\[-0.3cm]
\end{array}
\]
\end{proof}

The main goal of this section is to establish the existence of $E$-minimizing curves joining two given points in $\R^N$. Here we present a preliminary result on the existence of minimizers joining two nearby points, both far away from the zero set $\Z$ of $W$. There the metric $d$ is locally equivalent to the standard Euclidean metric. Consequently, the existence of $E$-minimizing curves joining nearby points will follow easily from an application of the direct method in the calculus of variations, where compactness is recovered from the non-degeneracy of the metric $d$ away from the wells.

\begin{lem}\zlabel{lem:local-existence-Emin}
For every $\e>0$ such that $\Z\subset B_{1/\e}(0)$ there exists a number $r_{\e}>0$ such that for all $p,q\in B_{1/\e}(0)\setminus \bigcup^m_{j=1}B_{2\e}(\p_j)$ satisfying $|p-q|<r_{\e}$, there exists an $E$-minimizing curve joining $p$ to $q$ that avoids an $\e$-neighborhood of $\Z$. \\[-.25cm]
\end{lem}

\begin{proof}[Proof of Lemma~\thmref{lem:local-existence-Emin}]
Given $\e>0$, we define the three positive numbers
\[
M_\e:=\max_{\{p:\,\abs{p}\leq 1/\e\}}F(p),\quad m_\e:=\min_{\bigcup_j \{p:\;\e\leq \abs{p-\p_j}\leq 2\e\}}F(p)\quad\mbox{and}\quad r_\e:=\frac{
\e \,m_\e}{M_\e}.
\]

Now we take distinct points $p,q\in B_{1/\e}(0)\setminus \bigcup^m_{j=1}B_{2\e}(\p_j)$ satisfying $|p-q|\leq r_\e$ and let $\{\gamma_k\}\subset Lip_{\Z}(p,q)$ denote a minimizing sequence in \eqref{Emini} so that $E(\gamma_k)\to d(p,q).$ Denoting by $\alpha_{\text{aff}}$ the line segment joining $p$  to $q$ we may assume $E(\gamma_k)\leq E(\alpha_{\text{aff}})$ so that we obtain the upper bound
\begin{equation}
E(\gamma_k)\leq E(\alpha_{\text{aff}})=\int_{0}^1F(\alpha_{\text{aff}}(s))|p-q|ds\leq M_\e\,r_\e.\label{lesslinear}
\end{equation}
In light of \eqref{lesslinear}, we now claim that for each $k$, $\gamma_k$ cannot pass within $\e$-Euclidean distance of the zero set of $F$. To see this, note
that if it did pass within $\e$ of $\p_j$ for some $j$, then we would have the lower bound
\[
E(\gamma_k)\geq \int\limits_{\{t:\;\e\leq \abs{\gamma_k(t)-{\bf{p}}_j}\leq 2\e\}}F(\gamma_k(t))|\gamma_k'(t)|\,dt\\\geq 2\e\,m_\e,\]
which is impossible given the definition of $r_\e$, see Figure~\ref{fig:1} below.
\begin{figure}[!ht]
  \centering
   \includegraphics[width=0.6\textwidth]{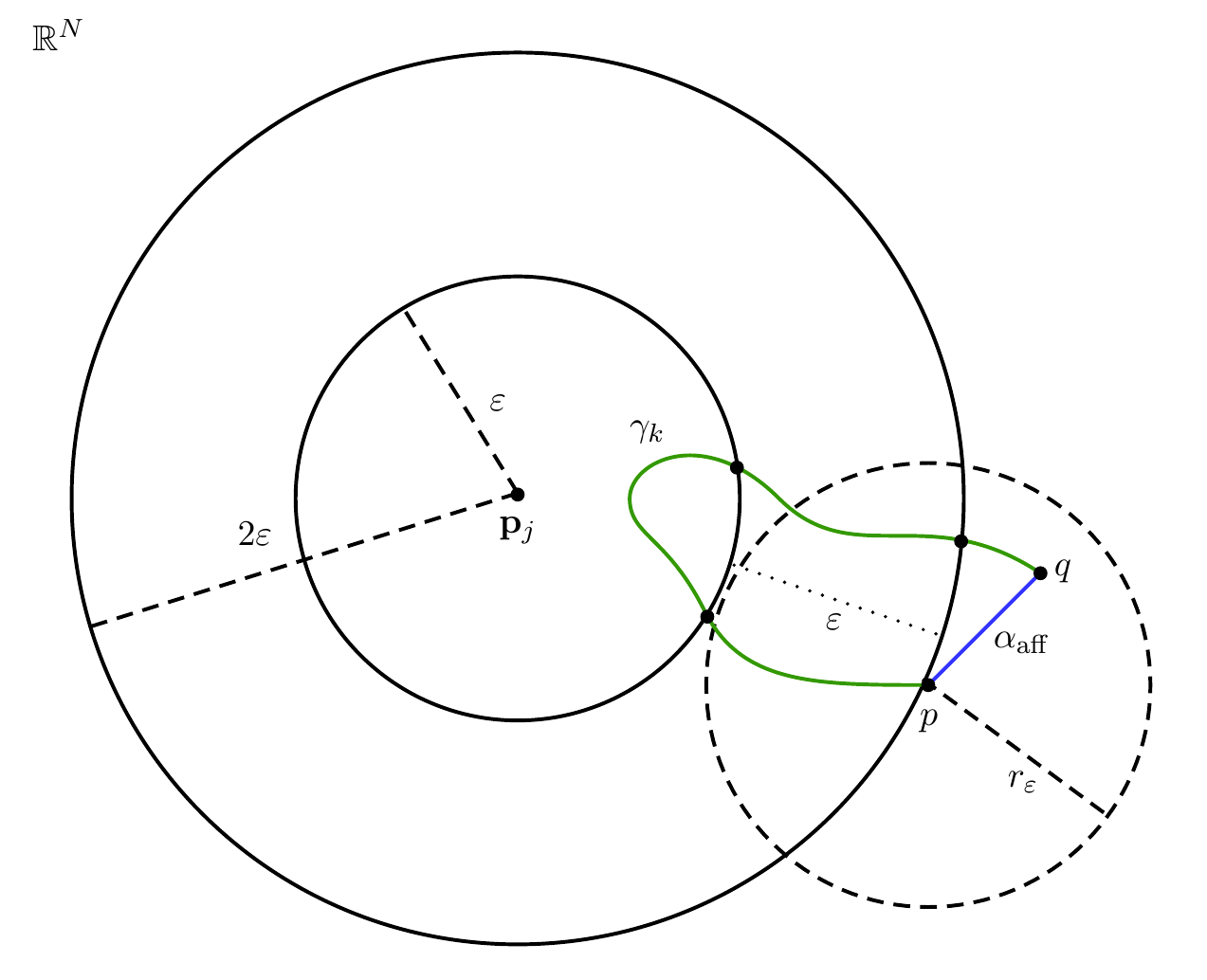}
\caption{Example of $\gamma_k\in Lip_{\Z}(p,q)$ with $\gamma_k([0,1])\cap B_{\varepsilon}(\p_j)\neq \emptyset$.}
\label{fig:1}
\end{figure}

Once we know each $\gamma_k$ avoids an $\e$ neighborhood of $\Z$, we can invoke the assumption \assref{assdecay} to easily obtain the compactness we need, in that for all $k$ we have
\[
\int^1_{0}\abs{\gamma'_k}dt\leq \left(\min\limits_{\bigcup_j \{p:\,\abs{p-\p_j}\geq \e\}}F(p)\right)^{-1}\int^1_{0}F(\gamma_k)\abs{\gamma_k'}dt
\leq \left(\min\limits_{\bigcup_j \{p:\,\abs{p-\p_j}\geq \e\}}F(p)\right)^{-1}(d(p,q)+1).
\]

We can then reparametrize $\gamma_k$ by constant speed $C_k:=\int^1_{0}\abs{\gamma_k'}dt$ with $C_k$ bounded independent of $k$. Hence, $\gamma_k:[0,1]\to \R^N$ are equi-Lipschitz. Applying the Arzela-Ascoli theorem, we have that upon extraction of a subsequence, there exists a Lipschitz curve $\gamma_0$ such that $\gamma_{k_\ell}\rightrightarrows \gamma_0$ with respect to the Euclidean metric. The lower-semi-continuity result Lemma~\thmref{lem:lsc-E-functional} then shows 
that
\[ 
d(p,q)=\liminf_{\ell\to\infty}E(\gamma_{k_\ell})\geq E(\gamma_0),
\]
and so $\gamma_0$ minimizes $E$.
\end{proof}

The key result of this subsection is the assertion below that $E(\gamma)$ and $L(\gamma)$ coincide:

\begin{thm}\zlabel{thm:EequalsL}
For any curve $\gamma\in Lip_{\Z}([0,1];\R^N)$ one has the equivalence 
\[
L(\gamma)=E(\gamma).
\]
\end{thm}
\begin{proof}[Proof of Theorem~\thmref{thm:EequalsL}]
In light of the additivity property enjoyed by both functionals $E$ and $L$ on concatenation of curves, given any $\gamma\in Lip_{\Z}([0,1];\R^N)$, it will suffice to establish this equivalence on any arc of $\gamma$ such that $\gamma$ avoids the set $\Z$ except perhaps at one or both of its endpoints. This follows since clearly an arbitrary curve can be decomposed into a union of such arcs. Pursuing the worst case scenario, with a slight abuse of notation we will then simply assume that $\gamma$ is one such arc with endpoints lying on distinct points $\p_{k_1}$ and $\p_{k_2}$ of $\Z.$


We first observe that for any partition $\{s_j\}^{J}_{j=0}$ of $[0,1]$ one has using the definition of the metric $d$ on each arc $\gamma([s_j,s_{j+1}])$
that
\[
\sum^{J-1}_{j=0}d(\gamma(s_j),\gamma(s_{j+1}))
= \sum^{J-1}_{j=0}\inf_{\beta\in Lip(\gamma(s_j),\gamma(s_{j+1}))}E(\beta)
\leq \sum^{J-1}_{j=0}E(\gamma|_{[s_j,s_{j+1}]} )=E(\gamma).
\]
Taking the supremum over all partitions $P([0,1])$ we conclude that $L(\gamma)\leq E(\gamma)$. 
\smallskip

We now turn to the task of proving the reverse inequality. We remark that in the case the curve satisfies $L(\gamma)=\infty$, then it immediately follows $E(\gamma)=\infty$ and so $L(\gamma)=E(\gamma)$. Let us then assume that $L(\gamma)<\infty$. We will argue in this case that $E(\gamma)$ is finite and further $E(\gamma)\leq L(\gamma)$. The strategy is to construct a sequence $\{\gamma_n\}\subset Lip_{\Z}([0,1];\R^N)$ with $n\to\infty$ that interpolates on a set of $n$ points on $\gamma$, where the interpolating pieces are chosen to be $E$-minimizing. 

 First choose $\e>0$ sufficiently small so that $\gamma((0,1))\subset B_{1/\e}(0)$. Then define $0<\underline{t}_{\e}<\overline{t}_{\e}<1$ through
 the conditions
 \[
 \underline{t}_{\e}:=\max\{t: \abs{\gamma(t)-\p_{k_1}}\leq2\e\},\quad \overline{t}_{\e}:=\min\{t:\abs{\gamma(t)-\p_{k_2}}\leq 2\e\}.
 \]
The claim is that $\underline{t}_{\e}\to 0$ and $\overline{t}_{\e}\to 1$ as $\e \to 0^+$. By contradiction let us suppose that there exists a sequence $\{\underline{t}_{\e_j}\}$ with $\underline{t}_{\e_j}\to t_*\in (0,1)$ as $\e_j\to 0^+$. It follows that $|\gamma(\underline{t}_{\e_j})-\p_{k_1}|=\e_j$ for every $j$ by definition of $\underline{t}_{\e_j}$. Taking the limit $j\to\infty$ in the above equality we deduce by the continuity of $\gamma$ that $|\gamma(t_*)-\p_{k_1}|=0$ and so $\gamma(t_*)=\p_{k_1}$, but this contradicts the assumption that the set $\Z$ is avoided by $\gamma$ at intermediate times. The second statement of the claim can be argued similarly.
 
For all $\e>0$ sufficiently small, one has $\gamma([\,\underline{t}_{\e},\overline{t}_{\e}])\subset B_{1/\e}(0)\setminus \bigcup^m_{i=1}B_{2\e}(\p_i)$.
  Hence,
  Lemma~\thmref{lem:local-existence-Emin} yields the existence of $r_{\e}>0$ so that any choice of points $p,q\in \gamma([\underline{t}_{\e},\overline{t}_{\e}])$ with $|p-q|<r_{\e}$ can be joined with an $E$-minimizing curve.

  For all $n$ sufficiently large, we now label $p_1:=\gamma(\underline{t}_{\e})$ and $p_{n-1}:=\gamma(\overline{t}_{\e})$ and then pick a collection of equi-spaced points $\{p_2,\ldots, p_{n-2}\}\subset \gamma([\underline{t}_{\e},\overline{t}_{\e}])$ so that $|p_k-p_{k+1}|<r_\e$ for $j=1,\ldots, n-2$. Let us write $p_k=\gamma(t_k)$ with increasing $t$-values. By virtue of Lemma~\thmref{lem:local-existence-Emin} on every pair of consecutive points we find $E$-minimizing curves, that up to reparametrization we write as $\alpha_k\in Lip([t_k,t_{k+1}];\R^N)$ with $\alpha_k(t_k)=p_k$ and $\alpha_k(t_{k+1})=p_{k+1}$ for $j=1,\ldots,n-2$. 
We define $\gamma_n\in Lip_{\Z}([0,1];\R^N)$ as the concatenation of $\alpha^-_{\text{aff}},\alpha_{1},\ldots,\alpha_{n-2},\alpha^+_{\text{aff}}$, where $\alpha^-_{\text{aff}}:[0,\underline{t}_{\e}\,]\to\R^N, \alpha^+_{\text{aff}}:[\,\overline{t}_{\e},1]\to\R^N$ are parametrizations of the linear segments $[\p_{k_1},p_1]$ and $[p_{n-1},\p_{k_2}]$, respectively, see Figure~\ref{fig:2} below.
\begin{figure}[!ht]
\centering
 \centering
  \includegraphics[width=0.8\textwidth]{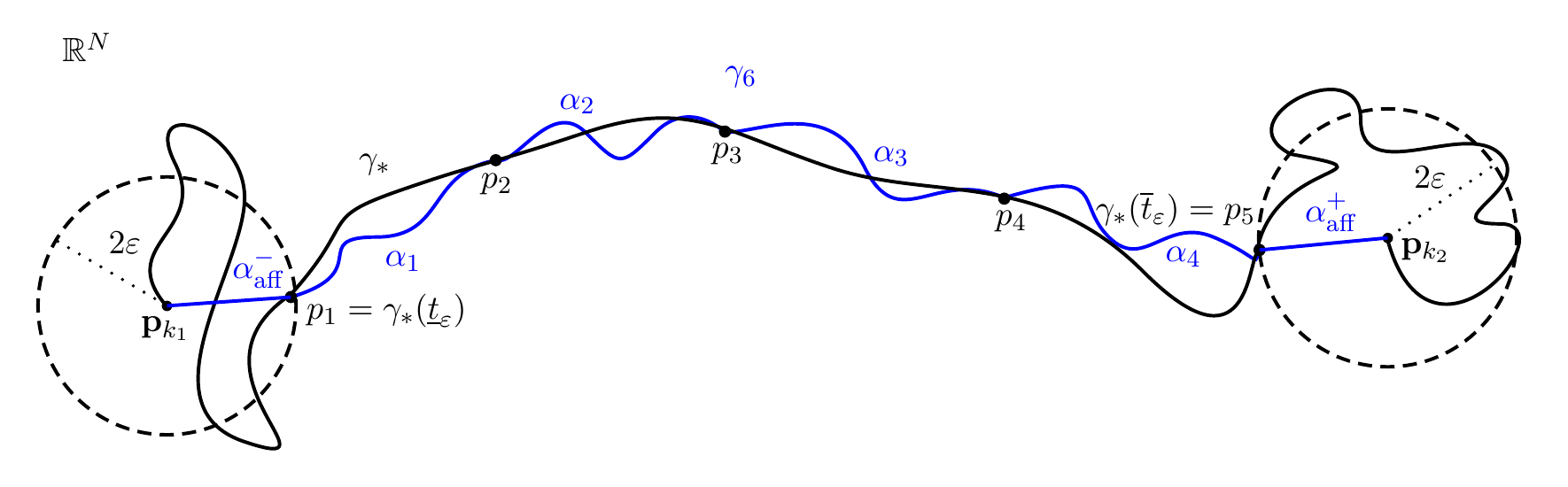}
\caption{Construction of the interpolating sequence $\{\alpha_k\}^{n-2}_{k=1}$, in the case $n=6$.}
\label{fig:2}
\end{figure}
\medskip

One readily checks that as $n\to\infty$ we have $|\gamma_n(t)-\gamma(t)|\rightrightarrows 0$  uniformly for $t\in[\,\underline{t}_{\e},\overline{t}_{\e}]$. Then Lemma~\thmref{lem:lsc-E-functional} applied on $[\,\underline{t}_{\e},\overline{t}_{\e}]$ yields
\begin{equation}\label{EminEqn1}
E(\gamma|_{[\,\underline{t}_{\e},\overline{t}_{\e}]})\leq \liminf_{n\to\infty}E(\gamma_n).
\end{equation}
On the other hand, since $\{\alpha_k\}^{n-2}_{k=1}$ are $E$-minimizers among Lipschitz continuous competitors with fixed endpoints we get
\[
 \sum^{n-2}_{k=1}E(\alpha_k)=\sum^{n-2}_{k=1}d(p_k,p_{k+1})=\sum^{n-2}_{k=1}d(\gamma(t_k),\gamma(t_{k+1}))\leq L(\gamma|_{[\,\underline{t}_{\e},\overline{t}_{\e}]}).
\] 
We conclude that for all large $n$ one has
\[
E(\gamma_n)=E(\alpha^-_{\text{aff}})+\sum^{n-2}_{k=1}E(\alpha_k)+E(\alpha^+_{\text{aff}})
\leq L(\gamma)+4M\e,
\]
where $M:=\max\{F(p):\,p\in \overline{B_{2\e}(\p_{k_1})}\cup \overline{B_{2\e}(\p_{k_2})}\}$, using the trivial bound on the two linear pieces:
\[
E(\alpha^{\pm}_{\text{aff}})\leq 2M\e.
\]
(We note that in fact $M=M_\e\to 0$ as $\e\to 0$ but we won't need this.)
 Hence,
\begin{equation}\label{EminEqn3}
\limsup_{n\to\infty}E(\gamma_n)\leq L(\gamma)+4M\e.
\end{equation}
Applying the estimates \eqref{EminEqn1} and \eqref{EminEqn3}, we conclude that
\[ 
E(\gamma|_{[\,\underline{t}_{\e},\overline{t}_{\e}]})\leq L(\gamma)+4M\e.
\]
Finally, by the claim above $\underline{t}_{\e}\to 0$ and $\overline{t}_{\e}\to 1$ as $\e\to 0$, so one has 
$E(\gamma|_{[\,\underline{t}_{\e},\overline{t}_{\e}]})\to E(\gamma)$ as $\e\to 0$. Consequently, 
\[ 
E(\gamma)\leq L(\gamma).
\]
Thus, $E(\gamma)= L(\gamma)$ for every $\gamma\in Lip_{\Z}([0,1];\R^N)$.
\end{proof}

\medskip

\subsection{Existence of minimizing geodesics}

Now we turn to the existence of $E$-minimizing curves joining any two points in $\R^N$, including the case of joining two wells of the potential, where the conformal factor $F:=\sqrt{W}$ of $d$ degenerates. 

\begin{thm}\zlabel{thm:existence-length-min}Given two distinct points $p,q\in \R^N$, there exists a curve $\gamma_*\in Lip_{\Z}(p,q)$  satisfying 
\begin{equation}
L(\gamma_*)=\inf_{\gamma\in Lip_{\Z}(p,q)}L(\gamma)=d(p,q)=E(\gamma_*).\label{geo}
\end{equation}
\end{thm}

\bigskip
We will refer to such a $\gamma_*$ as a \emph{minimizing geodesic between $p$ and $q$}. The proof is an application of the direct method in the context of length spaces, and is an adaptation of a proof of the Hopf-Rinow theorem for length spaces, see e.g.~\cite[pp. 35]{bridson1999metric}.

\smallskip

\begin{proof}[Proof of Theorem~\thmref{thm:existence-length-min}] The proof consists of demonstrating the existence of a curve in $Lip_{\Z}(p,q)$ that minimizes the length functional $L$. That such a minimizer also is a minimizer of $E$ then follows immediately from  Theorem \thmref{thm:EequalsL}. We will break the proof into two steps.
\medskip

{\bf Step 1. }We first exhibit a curve $\gamma_*$ with minimal length having endpoints $p$ and $q$. By definition of the metric, there is a minimizing sequence $\{\gamma_j\}\subset Lip_{\Z}(p,q)$  with $E(\gamma_j)\to d(p,q)$ as $j\to\infty$. Then by Theorem \thmref{thm:EequalsL},  $\lim_{j\to\infty}L(\gamma_j)=d(p,q)$ as well. With no loss of generality we assume $F(\gamma_j(t))>0$ for every $t\in(0,1)$ and all $j$, since if $F(\gamma_{j_0}(t_0))=0$ for some $t_0\in (0,1)$, then we can locally modify the curve $\gamma_{j_0}$ around $t_0$ so that it satisfies the above property, with $E(\gamma^{\text{mod}}_{j_0})\leq E(\gamma_{j_0})+1/j_0$. 
We now claim that appropriately parametrized, the sequence $\{\gamma_j\}$ is equi-Lipschitz in $(\R^N,d)$. For each $j$:
\begin{itemize}
\item First we reparametrize $\gamma_j$ using \emph{degenerate arc-length} 
\[ 
F(\gamma_j(\ell))|\gamma'_j(\ell)|=1, \qquad 0\leq \ell \leq E(\gamma_j).
\]
\item Then we normalize the time-scale using the energy, by introducing the parameter
\[ 
\tilde{\ell}\equiv \ell/ E(\gamma_j),\,\mbox{ so that }\, 0\leq \tilde{\ell}\leq 1.
\]
\end{itemize}
Working with this parameter, we write
\[ 
\tilde{\gamma}_j:[0,1]\to\R^N\quad\text{ with }\quad\tilde{\gamma}_j(\tilde{\ell})\equiv \gamma_j(\ell).
\]
 Fixing any $0\leq \tilde{\ell}_1\leq \tilde{\ell}_2\leq 1$, for any given $\delta>0$ it holds that $E(\gamma_j)\leq d(p,q)+\delta$, provided $j \geq j_0(\delta)$ large enough. Then we have the lower estimate
\[
\begin{array}{rl}
\tilde{\ell}_2-\tilde{\ell}_1=&
{\displaystyle \frac{\ell_2-\ell_1}{E(\gamma_j)}=\frac{1}{E(\gamma_j)}\int^{\ell_2}_{\ell_1}F(\gamma_j)|\gamma'_j|\, d\ell
\geq\frac{1}{(d(p,q)+\delta)}E(\gamma_j|_{[\ell_1,\ell_2]})}\\[0.5cm]
\geq&{\displaystyle \frac{1}{(d(p,q)+\delta)}d(\gamma_j(\ell_1),\gamma_j(\ell_2))
=\frac{1}{(d(p,q)+\delta)}d(\tilde{\gamma_j}(\tilde{\ell}_1),\tilde{\gamma}_j(\tilde{\ell}_2)).}
\end{array}
\]
Hence for all $j\geq j_0(\delta)$ we have
\[ 
d(\tilde{\gamma_j}(\tilde{\ell}_1)),\tilde{\gamma}_j(\tilde{\ell}_2))
\leq (d(p,q)+\delta)|\tilde{\ell}_1-\tilde{\ell}_2|,
\]
which shows that the family $\{\tilde{\gamma}_j\}$ is equi-Lipschitz with respect to the $d$ metric.  In addition, this family is uniformly bounded in the $d$ metric, since for any $\delta$ small one has
\[
\sup_{\tilde{\ell}\in [0,1]}d(p,\tilde{\gamma}_j(\tilde{l}))=\sup_{\tilde{\ell}\in [0,1]}d(\tilde{\gamma}_j(0),\tilde{\gamma}_j(\tilde{\ell}))\leq d(p,q)+1,
\]
which implies $\cup_{j\geq j_0(\delta)}\tilde{\gamma}_j([0,1])\subset \{\tilde{p}\in \R^N:\,d(p,\tilde{p})\leq d(p,q)+1\}$.
The Arzela-Ascoli theorem then ensures the existence of a subsequence $\{\tilde{\gamma}_{j_n}\}$ converging $d$-uniformly to some curve $\gamma_*$ that is Lipschitz in the $d$-metric. In particular, $d(p,q)+\delta$ is a bound for the Lipschitz constant of $\gamma_*$, but since $\delta>0$ is arbitrary we arrive to the same conclusion with the sharpest bound $d(p,q)$.

Now we check that $\gamma_*$ is a length-minimizing curve between $p$ and $q$. Recall that $\{\tilde{\gamma}_{j_n}\}$ is a minimizing sequence with $\tilde{\gamma}_{j_n}\rightrightarrows\gamma_*$ in the $d$-metric. Then Lemma~\thmref{lem:lsc-length-functional} and Theorem~\thmref{thm:EequalsL} apply to yield
\[ 
d(p,q)\leq L(\gamma_*)\leq 
\liminf_{n\to\infty}L(\tilde{\gamma}_{j_n})= \liminf_{n\to\infty}E(\tilde{\gamma}_{j_n})=d(p,q).
\]
Since by Theorem \thmref{thm:EequalsL} we have $L(\gamma_*)=E(\gamma_*)$ this will complete the proof of \eqref{geo} once we show that $\gamma_*$ lies in $Lip_{\Z}(p,q)$, namely that its restriction to any sub-arc that avoids $\Z$ is Lipschitz continuous with respect to the Euclidean metric.

\bigskip

{\bf Step 2.}
To this end, consider the arc $\gamma_*([a,b])$ for any $(a,b)\subset\subset [0,1]$ such that $\gamma_*\big([a,b]\big)\cap\Z=\emptyset.$
With an eye towards applying  Lemma~\thmref{lem:local-existence-Emin}, we now fix $\e$ sufficiently small so that $\gamma_*([a,b])\subset B_{1/\e}(0)\setminus\bigcup^m_{j=1}B_{2\e}(\p_j)$. 
Then for any interval $[a_1,b_1]\subset [a,b]$, we consider a partition $a_1=t_0<t_1<t_2<\ldots<t_n=b_1$ such that for each $j$ we have $\abs{\gamma(t_{j+1})-\gamma(t_j)}<r_\e$, with $r_\e$ taken from that lemma. 
We get a collection $\{\alpha_j\}^n_{j=1}$ of $E$-minimizing curves joining the endpoints $\gamma_*(t_j)$ to $\gamma_*(t_{j+1})$ for $j=0,1,\ldots,n-1$, and we know that they all avoid an $\e$-neighborhood of $\Z$. 
Thus, for each $j$ we have
\begin{equation*}
\begin{array}{rl}
{\displaystyle d\left(\gamma_*(t_j),\gamma_*(t_{j+1})\right)=E(\alpha_j)=\int^{t_{j+1}}_{t_j}F(\alpha_j)|\alpha'_j|}
&\geq  {\displaystyle \left(\min_{\bigcup_j \{p:\,\abs{p-\p_j}\geq \e\}}F(p)\right) \int^{t_{j+1}}_{t_j}|\alpha'_j|} \\[0.5cm]
&\geq {\displaystyle \left(\min_{\bigcup_j \{p:\,\abs{p-\p_j}\geq \e\}}F(p) \right)|\gamma_*(t_{j+1})-\gamma_*(t_{j})|}.
\end{array}
\end{equation*}
But since $d\left(\gamma_*(t_j),\gamma_*(t_{j+1})\right)\leq L_d\big(t_{j+1}-t_j\big)$ for each $j$, where $L_d$ denotes the Lipschitz constant of $\gamma_*$ with respect to the $d$ metric, we can sum over $j$ to find that
\[
\abs{\gamma_*(b_1)-\gamma_*(a_1)}\leq \displaystyle \left(\min_{\bigcup_j \{p:\,\abs{p-\p_j}\geq \e\}}F(p) \right)^{-1}L_d\big(b_1-a_1\big),
\]
and so $\gamma_*\in Lip_{\Z}([0,1];\R^N)$.  
\end{proof}


In passing, we mention some basic properties of minimizing geodesics.

\begin{prop}\zlabel{prop:properties-length-min} For any two distinct points $p,q\in \R^N$, consider a length-minimizing curve $\gamma_*\in Lip_{\Z}(p,q)$.
\begin{enumerate}
\item The restriction of $\gamma_*$ to any arc is length-minimizing: For any $[a,b]\subset [0,1]$,
\[ 
L(\gamma_*|_{[a,b]})=d(\gamma_*(a),\gamma_*(b)).
\]\\[-0.5cm]

\item The length of $\gamma_*$ is achieved by computing over any finite partition $\mathcal{P}_0$ of $[0,1]$,
\[ 
L(\gamma_*)=\sum_{t_i\in\mathcal{P}_0}d(\gamma_*(t_i),\gamma_*(t_{i+1})).
\] 
\end{enumerate}
\end{prop}
\medskip
\begin{proof}[Proof of Proposition~\thmref{prop:properties-length-min}]
For the first point, if $\gamma_*$ were not length-minimizing in some subinterval 
$[a_0,b_0]$ then $d(\gamma_*(a_0),\gamma_*(b_0))<L(\gamma_*|_{[a_0,b_0]})$, but this directly contradicts the length-minimality of the entire curve. Indeed, using the triangle inequality, and the additivity of the length functional on concatenations of curves one has
\[
\begin{array}{rl}
d(\gamma_*(0),\gamma_*(1))\leq& 
d(\gamma_*(0),\gamma_*(a_0))+d(\gamma_*(a_0),\gamma_*(b_0))+d(\gamma_*(b_0),\gamma_*(1))\\[0.3cm]
< & {\displaystyle L(\gamma_*|_{[0,a_0]})+L(\gamma_*|_{[a_0,b_0]})+L(\gamma_*|_{[b_0,1]})=L(\gamma_*).}
\end{array}
\]
This proves the length-minimality of any arc $\gamma_*([a,b])$ for every $[a,b]\subset[0,1]$.

For the second point, given any partition $\mathcal{P}_0=\{t_i\}^{I+1}_{i=0}$ of $[0,1]$, we use the length-minimality result just proved on every arc 
$\gamma_*([t_i,t_{i+1}])$. We deduce
\begin{align*}
L(\gamma_*)=
\sum^I_{i=0} L(\gamma_*|_{[t_i,t_{i+1}]})
=\sum^I_{i=0}d(\gamma_*(t_i),\gamma_*(t_{i+1})),
\end{align*} 
in light of the $L$-additivity on concatenations of curves. 
\end{proof}
\smallskip

Finally we address the regularity of minimizing geodesics under further smoothness assumptions on $F$  away from $\Z$, or equivalently, on $W$. 

\begin{prop}\zlabel{prop:regularity} In addition to assumptions \emph{\assref{asswells}} and  \emph{\assref{assdecay}}, assume  that $F\in C^{1,\alpha}_{loc}(\R^N\setminus\Z)$ with $\alpha\in(0,1)$. Then for distinct points $\p_j,\p_k\in \Z$, a minimizing geodesic $\gamma_*$ joining $\p_j$ to $\p_k$ admits a $C^{2,\alpha}$-parametrization along any connected arc that avoids $\Z$.
\end{prop}

\medskip

\begin{rem}
If we further assume that $F\in C^{k+1,\alpha}_{loc}(\R^N\setminus \Z)$ for some $k\geq 1$ and $\alpha\in(0,1)$, then the same conclusion in Proposition~\thmref{prop:regularity} holds with $C^{k+2,\alpha}$-parametrizations.
\end{rem}
\smallskip
\begin{proof}[Proof of Proposition~\thmref{prop:regularity}] 
Fix any interval $[a,b]\subset (0,1)$ such that $\gamma_*([a,b])\cap \Z=\emptyset.$
The regularity established in Step 2 of Theorem~\thmref{thm:existence-length-min} yields that this arc of $\gamma_*$ is Lipschitz continuous with respect to the Euclidean metric. Hence, by the Rademacher's Theorem $\gamma'_*$ exists a.e. in $[a,b]$. We note that if $\gamma'_*(t)= 0$ for a.e. $t$ in some interval $I\subset [a,b]$, then we can simply reparametrize this curve by removing $I$ and so with no loss of generality we may assume that  $|\gamma'_*(t)|\neq 0$ for a.e. $t\in [a,b]$. 
Now we can reparametrize $\gamma_*$ restricted to $[a,b]$ by a multiple of Euclidean arc-length, choosing $s=s(t):=(1/l)\int^t_a |\gamma'_*(\tau)|d\tau$ with $l=\int^b_a|\gamma'_*(t)|dt$ so that in the new parametrization of this arc one has $\gamma_*:[0,1]\to\R^N$ satisfying
\[
\abs{\dfrac{d\gamma_*}{ds}(s)}\equiv l\quad\mbox{for a.e}\;\; s\in (0,1).
\] 

Taking an arbitrary curve $\gamma\in Lip([0,1];\R^N)$ satisfying $\gamma(0)=\gamma(1)=0$, we find that the arc of the minimizing geodesic $\gamma_*$ under consideration satisfies the criticality condition
\[
\begin{array}{rl}
0=\delta E(\gamma_*;\gamma)
=&{\displaystyle \frac{d}{d\lambda}\bigg|_{\lambda=0}\int^{1}_{0}F(\gamma_*+\lambda\gamma)\abs{\frac{d\gamma_*}{ds}+\lambda\frac{d\gamma}{ds}}ds}\\[0.5cm]
=&{\displaystyle \int^{1}_{0}\{(1/l)F(\gamma_*)\frac{d\gamma_*}{ds}\cdot\frac{d\gamma}{ds}+l\nabla_p F(\gamma_*)\cdot\gamma\} ds,}
\end{array}
\]
and so is a weak solution of the Euler-Lagrange equation
\[
\dfrac{d}{ds}(F(\gamma_*)\dfrac{d\gamma_*}{ds})=l^2\nabla_p F(\gamma_*) \;\;\mbox{ on }\,\, (0,1).
\]
Now we may apply standard regularity theory. Since $F(\gamma_*)$ is a positive, Lipschitz continuous
function on $(0,1)$ and the right-hand side lies in $L^2((0,1);\R^n)$, we conclude that $\gamma_*\in H^{2}_{loc}((0,1);\R^N)$ (see \cite{gilbarg2001elliptic}, Thm.8.8). Hence, by Sobolev embedding we have $\gamma_*\in C^{1,\alpha}_{loc}((0,1);\R^N)$.
  But then, we can view $\gamma_*$ as a weak solution of the following differential equation
\[
F(\gamma_*)\dfrac{d^2\gamma_*}{ds^2}=-(\nabla_p F(\gamma_*)\cdot \dfrac{d\gamma_*}{ds})\dfrac{d\gamma_*}{ds}+l^2\nabla_pF(\gamma_*)=:G \quad \mbox{ on }\quad (0,1),
\]
with $G\in C^{0,\alpha}_{loc}((0,1);\R^N)$. It then immediately follows that
$\gamma_*\in C^{2,\alpha}_{loc}((0,1);\R^N)$.

In case $F\in C^{k+1,\alpha}_{loc}(\R^N\setminus \Z)$,  a standard bootstrap argument allows one to deduce that $\gamma_*\in C^{k+2,\alpha}_{loc}((0,1);\R^N)$ for $k\geq 2$.
\end{proof}

\bigskip

\section{Existence of heteroclinic connections}\plainsection\zlabel{sec:study-connection-problem}
\subsection{Existence of heteroclinics through re-parametrization of minimizing geodesics}

In this section we establish the existence of heteroclinic connections between two wells of the multi-well potential $W=F^2$. Throughout, we make the following assumptions: 
\medskip
\begin{itemize}
\item[\assref{asswells}] The zero set $\Z$ of $W$ is given by $\Z=\{\p_1,\ldots \p_m\}$ so that $W(\p_1)=\ldots =W(\p_m)=0$, and $W>0$ elsewhere.\\[-0.2cm]
\item[\assref{assdecay}] $\liminf_{|p|\to\infty}W(p)>0$.\\[-0.2cm]
\item[\asslabel{assregularity}] $W\in C^{1,\alpha}_{loc}(\R^N\setminus\Z)$.\\[-0.2cm]
\item[\asslabel{assgrowthbdd}] There are positive numbers $C$ and $\delta$ such that 
$W(p)\leq C|p-\p_j|^2$, for $|p-\p_j|<\delta$ and $j=1,\ldots, m$.
\end{itemize}
\bigskip

\smallskip

 For the study of heteroclinic connections we re-introduce the functional  
$
H:H^1_{loc}(\R,\R^N)\to\R$ given by
\[ 
H(U) :=\int^{+\infty}_{-\infty}\frac{1}{2}\abs{U'}^2+W(U). 
\]

 Following the scheme developed in \cite{sternberg1991vector}, we study the connection problem by carrying out an analysis of the link between the geometric problem 
\[
\mbox{(GP)} \hspace{1.2cm} \inf\limits_{\gamma\in Lip_{\Z}(\p_j,\p_k)}\quad E(\gamma)
\]
that was the focus of Section 2, and the variational problem
\[
\mbox{(HP)} \hspace{0.7cm} \inf\limits_{\substack{U\in H^1_{loc}(\R,\R^N)\\ U(x)\to \p_j \text{ as }x\to -\infty\\U(x)\to \p_k \text{ as }x\to +\infty  }} H(U)
\]
whose critical points yield heteroclinic connections between two distinct wells $\p_j,\p_k\in\Z$.
 
We recall that the analysis of  Section 2 yielded solutions to (GP), namely the minimizing  geodesics of Theorem~\thmref{thm:existence-length-min}. We will argue that under an appropriate parametrization, these geodesics give rise to solutions to problem (HP) as well, though as we will see, the connected wells may differ from $\p_j$ and $\p_k$.
Hence, in particular they will satisfy the associated Euler-Lagrange equation, 
\[
\begin{array}{l} 
U''_*-\nabla_u W(U_*)=0\quad\mbox{ in }\;\; (-\infty,\infty),\\[0.25cm]
U_*(-\infty)=\p_j,\quad U_*(+\infty)=\p_k,
\end{array} 
\]
thus giving rise to a {\it heteroclinic connection} $U_*$ between two elements $\p_j$ and $\p_k$ of $\Z$. 
\medskip

The main result establishes the existence of an $H$-minimizing heteroclinic connection between two wells of $W$, provided the trajectory of some minimizing geodesic solving the geometric problem (GP) joining these wells does not visit some other zero of $W$ along the way.

\begin{thm}\zlabel{thm:existence-connection} For distinct points $\p_j,\p_k\in \Z$ consider a minimizing geodesic $\gamma_*\in Lip_{\Z}(\p_j,\p_k)$. Let us write $0=t_1<t_2<t_3<\ldots<t_J=1$ \emph{(}with $J\geq 2$\emph{)} for the times when $\gamma_*(t)\in\Z$, so that, in particular, $\gamma_*(t_1)=\p_j$ and $\gamma_*(t_J)=\p_k$. 
Then for every $i\in\{1,\ldots, J-1\}$ there exists an $H$-minimizing heteroclinic connection between the wells $\gamma_*(t_i)$ and 
$\gamma_*(t_{i+1})$.
\end{thm}

\begin{cor}\zlabel{cor:offZ} If $J=2$ in Theorem~\thmref{thm:existence-connection}, that is, if a minimizing geodesic connecting two zeros $\p_j$ and $\p_k$ of $W$ avoids any other zeros of $\, W$ along the way,  or equivalently, if the strict triangle inequality 
\[
d(\p_j,\p_k)<d(\p_j,\p_\ell)+d(\p_\ell,\p_k)\quad\mbox{holds for all }\;\p_\ell\in \Z\setminus\{\p_j,\p_k\},
\]
 then under an equipartition parametrization, this geodesic represents an $H$-minimizing connection between the two zeros.
\end{cor} 
\medskip

\begin{cor}[(Two-well case)]\zlabel{cor:two-well-case}
Assume the zero set $\Z$ of $W$ consists of exactly two points, $\p_1$ and $\p_2$.  Then there exists an $H$-minimizing heteroclinic connection between these wells.
\end{cor}
\medskip

In comparing Corollary~\thmref{cor:two-well-case} to the earlier results found in \cite{sternberg1991vector} and \cite{alikakos2008connection}, we point out that our assumptions \assref{asswells}-\assref{assgrowthbdd} are quite weak. For example, in \cite{sternberg1991vector} there is an assumption that the Hessian matrix of $W$ is positive definite at the two wells, while in \cite{alikakos2008connection}, the authors assume a radial monotonicity condition holds at the two wells of the form
\[
\exists \,r_0>0\;\mbox{such that}\;r\mapsto W(\p_{\pm}+r\xi)\;\mbox{is increasing for all}\;r\in (0,r_0)\;\mbox{and all}\;\xi\in \mathbb{S}^{N-1},
\]
cf. assumption (h) in \cite{alikakos2008connection}.
However, Corollary~\thmref{cor:two-well-case} above holds even for potentials that oscillate near the wells, such as \[
W(p)=\prod^2_{j=1}(2+\sin(1/|p-\p_j|))|p-\p_j|^2,\quad p\in \R^N,\quad \p_1,\p_2\in\R^N,
\]
that do not satisfy assumption ({\em h}). On the other hand, in the recent work \cite{byeon2016double} existence is obtained for a $C^1$ potential satisfying our assumptions \assref{asswells}-\assref{assregularity} without \assref{assgrowthbdd}.

Before beginning the proof of this proposition we remark on some properties of the so-called parametrization \emph{in equipartition of energy} which plays a crucial role in our approach, cf.(\cite{sternberg1991vector}, Lemma 1):

\begin{lem}\zlabel{lem:equip-energy} For any curve $\gamma\in C^0([0,1];\R^N)\cap C^1((0,1);\R^N)$ with non-vanishing derivative satisfying $\gamma(0)=\p_j,\;\gamma(1)=\p_k\,$ for $\,\p_j,\;\p_k\in \Z$ and $W(\gamma(t))> 0$ for all $t\in(0,1)$, define the equipartition parameter $x:(0,1)\to \R$ via
\begin{equation} \label{defEquipParam}
x(t):=
\int^t_{\frac{1}{2}}
\frac{\left|\gamma'(u)\right|}{\sqrt{2W(\gamma(u))}} du.
\end{equation}
Then $x:(0,1)\to\R$ is smooth, increasing and satisfies $x\big((0,1)\big)=\R$. Furthermore, the curve $U:\R\to\R^N$ given by
 $U(x):= \gamma(t(x))$ satisfies a pointwise equipartition of energy in the sense that
\[
\sqrt{2W(U)}\abs{\dfrac{dU}{dx}}=\frac{1}{2}\abs{\dfrac{dU}{dx}}^2+W(U), \quad\mbox{ for all }x\in\R.
\]
\end{lem}
\begin{proof}[Proof of Lemma~\thmref{lem:equip-energy}]
We will write $\gamma'=\dfrac{d\gamma}{dt}$. That $x(t)$ is smooth and increasing is obvious. To establish that the range is all of $\R$,
we first reparametrize $\{\gamma(t): 0< t<1\}$ by Euclidean arc-length, denoted by $s$, so that the equipartition parameter becomes  
\[
x(s)=\int^s_{s(1/2)}\frac{1}{\sqrt{2W(\gamma(\tau))}}d\tau, \quad\mbox{with}\quad s(1/2)=\int^{1/2}_0|\gamma'(t)|dt. 
\]
We will argue that $\lim_{s\to 0^+}x(s)=-\infty$. The fact that  $\lim_{s\to s(1)^+}x(s)=+\infty$ will follow similarly. 
To this end, we fix $\eta>0$ so that $\abs{\gamma(s)-\p_j}<\delta$ for all $s\in (0,\eta)$ where $\delta$ is the value from assumption \assref{assgrowthbdd}.
In light of the arc-length parametrization, we know that $\abs{\gamma(s)-\p_j}\leq s$.
Consequently, for $s\in (0,\eta)$ one has
\[
W(\gamma(s))\leq C\abs{\gamma(s)-\p_j}^2\leq C s^2.
\]
Then we find that
\[
\lim_{s\to 0^+}|x(s)|\geq \lim_{s\to 0^+} \int_s^\eta \frac{1}{\sqrt{2W(\gamma(\tau))}}\, d\tau\geq
\lim_{s\to 0^+} \frac{1}{\sqrt{2C}}\int_s^\eta \frac{1}{\tau}\,d\tau=\infty.
\]
 
 The third point is an easy consequence of the chain rule:
\[
\frac{1}{2}\abs{\dfrac{dU}{dx}}^2=
\frac{1}{2}\abs{\gamma'\frac{dt}{dx}}^2
=\frac{1}{2}\abs{\gamma'}^2\dfrac{2W(U)}{\left|\gamma'\right|^2}=W(U).
\]
\end{proof}

\begin{proof}[Proof of Theorem~\thmref{thm:existence-connection}]
We consider first the case where $W(\gamma_*(t))> 0$ for all $0<t<1$, so that $\p_j$ and $\p_k$ are the only zeros of $W$ traversed by the curve. By Proposition~\thmref{prop:regularity}(i) we may take a $C^{2,\alpha}$-parametrization of $\gamma_*$, defined on $(0,1)$, with non-vanishing derivative. Then applying Lemma~\thmref{lem:equip-energy}, we reparametrize this curve using the equipartition parameter \eqref{defEquipParam}, yielding $U_*(x):= \gamma_*(t(x))$ with $U_*:(-\infty,+\infty)\to \R^N$.
Since $E$ is invariant under reparametrization, the fact that $\gamma_*$ is a minimizing geodesic between $\p_j$ and $\p_k$, implies that $U_*$ solves the geometric problem (GP) as well, so that $E(U_*)=E(\gamma_*)=d(\p_j,\p_k)$. On the other hand, by virtue of the equipartition of energy
in Lemma~\thmref{lem:equip-energy}, we have $E(U_*)=H(U_*)$. Consequently, we deduce that for any
$U\in H^1_{loc}(\R,\R^N)$ satisfying $U(-\infty)=\p_j$, $U(+\infty)=\p_k$
\[ 
H(U_*)=E(U_*)\leq E(U)\leq H(U).
\]
Hence, $U_*$ is a global minimizer of problem (HP). In view of the $C^{2,\alpha}$-regularity of $U_*$ guaranteed by Proposition~\thmref{prop:regularity}, this curve is a classical solution to the Euler-Lagrange equation associated with $H$. We conclude that $U_*$ is a $H$-minimizing heteroclinic connection between $\p_j$ and $\p_k$. 

To handle the case where $\gamma_*((0,1))\cap \Z\neq\emptyset$, as previously remarked in Proposition~\thmref{prop:regularity} the geodesic $\gamma^*$ can only traverse $\Z$ finitely many times, so we write 
 $0=t_1<t_2<t_3<\ldots<t_J=1$ with $J\geq 2$ for the times when $\gamma_*(t)\in\Z$.
Fixing any $i\in \{1,\ldots,J-1\}$ and restricting to the arc $\gamma_*((t_i,t_{i+1}))$ which does not intersect $\Z$, we apply the previous case to conclude that $U_*(x):=\gamma_*(t(x))$ is an $H$-minimizing heteroclinic connection between $\gamma_*(t_i),\gamma_*(t_{i+1})\in \Z$, for $t\in(t_i,t_{i+1})$.
\end{proof}

\subsection{Remarks on the obstruction to existence of heteroclinic orbits}

As observed in \cite{alama1997stationary,alikakos2006explicit,alikakos2008connection} the presence of multiple wells may obstruct the existence of heteroclinic connections between two given wells $\p_j$ and $\p_k$ in $\Z$. In particular, utilizing complex variables techniques in the planar case $N=2$ under the assumption that $W(z)=\abs{f(z)}^2$ with $f$ holomorphic,  the authors of \cite{alikakos2008connection} obtain various conditions, both necessary and sufficient, for existence of heteroclinic connections. For example, in case $W(z)$ takes the form 
\[
W(z)=\abs{\prod\limits_{j=1}^3 a(z-z_j)}^2,
\]
for some $a\in\C$, with $z_j\in\C$ taking the role of our $\p_j$, their Proposition 4.5 states that a heteroclinic connection exists between, say, $z_1$ and $z_2$ if and only if the strict triangle inequality holds for the metric $d$:
\[
d(z_1,z_2)<d(z_1,z_3)+d(z_3,z_2).
\]
 As a concrete example of the non-existence phenomenon, they consider the three-well potential $W_\e:\C\to \R$  given by $W_{\e}(z)=\abs{(1-z^2)(z-i\e)}^2$, for $\e\in\R$ and $z\in \C$, whose zero set is $\Z=\{-1,+1,i\e\}$. 
Their analysis proves that there exists a connection between $-1$ and $+1$ if and only if $|\e|>\sqrt{2\sqrt{3}-3}=:\e_*$. In particular, when
$\abs{\e}\leq \e_*$ they establish the identity $d(-1,+1)= d(-1,i\e)+d(i\e,+1)$, leading to the conclusion that the minimizing geodesic between
$-1$ and $1$ passes through the zero $i\e$, see the Figure~\ref{fig:3}.

\begin{figure}[!ht]
\centering
  \includegraphics[width=0.7\textwidth]{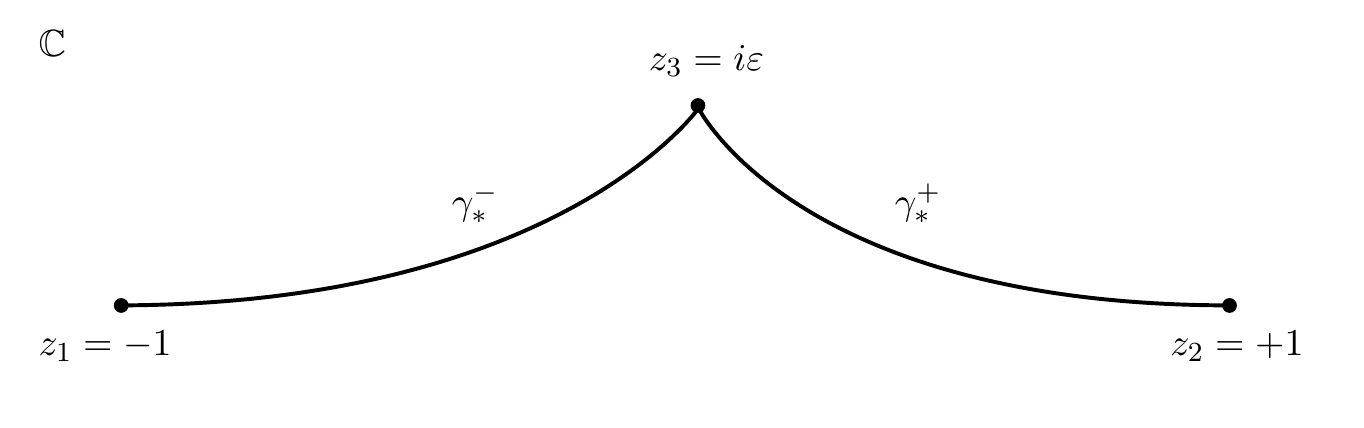}
\caption{Sketch of geodesics $\gamma^-_*\in Lip_{\Z}(-1,i\e),\, \gamma^+_*\in Lip_{\Z}(i\e,+1)$, for $0<\e\leq \e_*$.}
\label{fig:3}
\end{figure}
This illustrates the obstruction to be avoided in applying our Corollary~\thmref{cor:offZ}. Of course, in the generality in which our result holds, we only have the identity
\[
d(\p_j,\p_k)=d(\p_j,\p_\ell)+d(\p_\ell,\p_k)\quad\mbox{for some }\;\p_\ell\in \Z,
\]
as a necessary condition for non-existence of a heteroclinic connection between $\p_j$ and $\p_k$ in $\Z$.  Indeed, the striking rigidity of the Alikakos-Fusco result quoted above is surely related to the analyticity of $f$ in the assumption $W(z)=\abs{f(z)}^2$, in addition to its being in the planar setting. For example, we see no reason why {\it local} minimizers or even saddle points of $E$ or of $H$ should not exist for general $W:\R^N\to [0,\infty)$ having three or more zeros, leading to connections even when the (globally) minimizing geodesic fails to provide a heteroclinic connection because it passes through a third well.

Though we do not present an explicit example of such an occurrence, we will conclude with an example of a non-minimizing heteroclinic connection that co-exists with multiple minimizing geodesic connections.  Let us consider for $N=3$ the potential $W:\R^3\to\R$ given by
\[
W(x,y,z)
=x^2(1-x^2)^2+\left(y^2-\frac{1}{2}(1-x^2)^2\right)^2+\left(z^2-\frac{1}{2}(1-x^2)^2\right)^2.
\]
It can be readily checked that conditions (A1)-(A4) are satisfied by $W$. In particular we have that $m=6$, and the collection of zeros of $W$ consists of 
\begin{equation*}
\begin{array}{l}
\p_1=(-1,0,0),\,\, \p_2=(1,0,0),\\[0.2cm]
\p_3=\left(0,\frac{1}{\sqrt{2}},\frac{1}{\sqrt{2}}\right),\,\,
\p_4=\left(0,\frac{1}{\sqrt{2}},-\frac{1}{\sqrt{2}}\right),\,\,
\p_5=\left(0,-\frac{1}{\sqrt{2}},-\frac{1}{\sqrt{2}}\right),\,\,
\p_6=\left(0,-\frac{1}{\sqrt{2}},\frac{1}{\sqrt{2}}\right).
\end{array}
\end{equation*}

On the one hand, the existence of a heteroclinic connection between $\p_1$ and $\p_2$ can be argued by pursuing the ansatz $U_0(t)=(u(t),0,0)$ and then noting that the system of O.D.E.s $U''_0=\nabla W(U_0)$ reduces to a scalar differential equation 
\[
u''=-2u(1-u^2)(1+u^2+u^4)\,\,\mbox{ with }\,\,u(\pm\infty)=\pm 1.
\] 
Existence of such a solution $u$ follows from an elementary phase plane analysis and the use of the pointwise equi-partition relation $|U'_0|^2=2W(U_0(t))$. Hence, one gets the existence of a heteroclinic orbit $U_0:\R\to\R^3$ joining $\p_1$ to $\p_2$ that follows the $x$-axis. 

On the other hand, this line segment is not a minimizing geodesic between $\p_1$ and $\p_2$ since we can easily exhibit competitors with less $E$-values. For instance consider for $\mu\in[-1,1]$
\[
\gamma_{\pm,\pm}(\mu):=\left(\mu,\frac{\pm\e}{\sqrt{2}}(1-\mu^2),\frac{\pm\e}{\sqrt{2}}(1-\mu^2)\right),\quad \e=0.96.
\]
which yields four curves joining $\p_1$ to $\p_2$. If we let $\gamma_0(\mu)=(\mu,0,0)$ be a parametrization of the line segment, we explicitly compute
\begin{equation*}
\begin{array}{rl}
E(\gamma_{\pm,\pm})=& {\displaystyle \int^{1}_{-1}
\sqrt{W(\gamma_{\pm,\pm})}|\gamma'_{\pm,\pm}|\,d\mu}\\
=&{\displaystyle \int^{1}_{-1}
\sqrt{[\mu^2(1-\mu^2)^2+\frac{1}{2}(1-\mu^2)^4(1-\e^2)^2] [1+4\mu^2-4(1-\e^2)\mu^2] }\,d\mu}\\[0.4cm]
\approx &0.74,\displaybreak[3]\\[0.25cm]
\end{array}
\end{equation*}
while
\begin{equation*}
\begin{array}{rl}
E(\gamma_0)= & {\displaystyle \int^{1}_{-1}\sqrt{W(\gamma_0)}|\gamma_0'|\,d\mu}\\
=& {\displaystyle \int^{1}_{-1}\sqrt{\mu^2(1-\mu^2)^2+\frac{1}{2}(1-\mu^2)^4}\,d\mu}\\[0.4cm]
\approx & 0.98
\end{array}
\end{equation*}

It then follows that $U_0$ is not a global minimizer of $H$ either. To see this, note that since, for example, $\gamma_{+,+}:[-1,1]\to\R^3$ is a competing curve which does not run into any of the zeros $\Z=\{\p_1,\ldots,\p_6\}$ of $W$ other than at the endpoints, we can reparametrize it by the equipartition parameter, cf. Lemma~\zref{lem:equip-energy}, to get a new curve $U_1:(-\infty,\infty)\to\R^3$. Then the value of the two functionals $E$ and $H$ agree at this curve, so we conclude from the comparison of degenerate lengths above that 
\[
H(U_1)=E(U_1)=E(\gamma_{\pm,\pm})< E(\gamma_0)=E(U_0)\leq H(U_0). 
\]

In light of these facts and invariance of the potential $W$ under the reflections $y\mapsto -y,\;z\mapsto -z$, there will exist multiple minimizing geodesics joining $\p_1$ to $\p_2$ in addition to the non-minimizing heteroclinic connection along the $x$-axis.

\bigskip

\bibliographystyle{amsplain}

\bigskip

\end{document}